\documentclass [11pt,reqno]{amsart}
\usepackage {amsmath,amssymb,epsfig,verbatim,geometry}
\usepackage[all]{xy}

\usepackage{enumitem}

\usepackage[usenames,dvipsnames]{color}

\newcommand{\C}{\mathbb{C}}
\newcommand{\DD}{\mathbb{D}}

\newcommand{\Q}{\mathbb{Q}}
\newcommand{\R}{\mathbb{R}}
\newcommand{\Z}{\mathbb{Z}}
\newcommand{\N}{\mathbb{N}}
\renewcommand{\P}{\mathbb{P}}

\newcommand{\fa}{\mathfrak{a}}

\newcommand{\fS}{\mathfrak{S}}

\newcommand{\ddc}{dd^c}

\newcommand{\cC}{\mathcal{C}}

\newcommand{\cE}{\mathcal{E}}

\newcommand{\cH}{\mathcal{H}}
\newcommand{\cJ}{\mathcal{J}}
\newcommand{\cL}{\mathcal{L}}

\newcommand{\cN}{\mathcal{N}}
\newcommand{\cO}{\mathcal{O}}

\newcommand{\cX}{\mathcal{X}}

\newcommand{\lau}[1]{(\!(#1)\!)}

\renewcommand{\a}{\alpha}

\renewcommand{\d}{\delta}

\newcommand{\f}{\varphi}

\newcommand{\g}{\gamma}
\newcommand{\la}{\lambda}
\newcommand{\om}{\omega}
\newcommand{\p}{\psi}

\newcommand{\D}{\Delta}

\newcommand{\ie}{{\rm i.e.\ }}

\newcommand{\ddbar}{\partial\overline{\partial}}

\DeclareMathOperator{\MA}{MA}

\DeclareMathOperator{\id}{id}

\DeclareMathOperator{\ord}{ord}

\DeclareMathOperator{\Aut}{Aut}

\DeclareMathOperator{\PSH}{PSH}

\DeclareMathOperator{\Ric}{Ric}

\DeclareMathOperator{\spec}{Spec}

\DeclareMathOperator{\FS}{FS}
\DeclareMathOperator{\GL}{GL}

\DeclareMathOperator{\Herm}{Herm}

\newcommand{\CAT}{\mathrm{CAT}}
\newcommand{\NA}{\mathrm{NA}}

\newcommand{\ent}{\mathrm{ent}}

\newcommand{\pp}{\mathrm{pp}}

\numberwithin{equation}{section}       % Number formulas within sections

\newtheorem{prop} {Proposition} [section]
\newtheorem{thm}[prop] {Theorem}

\newtheorem{lem}[prop] {Lemma}
\newtheorem{cor}[prop]{Corollary}
\newtheorem{conj}[prop]{Conjecture}
\newtheorem{prop-def}[prop]{Proposition-Definition}

\newtheorem{exam}[prop]{Example}

\theoremstyle{remark}

\usepackage{color}
\usepackage{microtype}

\title[Variational and non-Archimedean]{Variational and non-Archimedean aspects of the Yau--Tian--Donaldson conjecture}
\date{\today}

\author{S{\'e}bastien Boucksom}

\thanks{The author was partially supported by the ANR project GRACK}

\address{CNRS-CMLS\\
  \'Ecole Polytechnique\\
  F-91128 Palaiseau Cedex\\
  France}
\email{sebastien.boucksom@polytechnique.edu}

\begin{document}

\begin{abstract} We survey some recent developments in the direction of the Yau-Tian-Donaldson conjecture, which relates the existence of constant scalar curvature K\"ahler metrics to the algebro-geometric notion of K-stability. The emphasis is put on the use of pluripotential theory and the interpretation of K-stability in terms of non-Archimedean geometry. 
\end{abstract}

\maketitle

\setcounter{tocdepth}{1}
\tableofcontents

%
%%%%%%%%%%%%%%%%%%%%%%%%%%%%%%%%%%%%%%%%%%%%%%%%%%%%%%%%%%%%%%%%%%%
%
\section*{Introduction} 
%
%%%%%%%%%%%%%%%%%%%%%%%%%%%%%%%%%%%%%%%%%%%%%%%%%%%%%%%%%%%%%%%%%%%
%
The search for constant curvature metrics is a recurring theme in geometry, the fundamental uniformization theorem for Riemann surfaces being for instance equivalent to the existence of a (complete) Hermitian metric with constant curvature on any one-dimensional complex manifold. On a higher dimensional complex manifold, \emph{K\"ahler metrics} are defined as Hermitian metrics locally expressed as the complex Hessian of some (plurisubharmonic) function, known as a \emph{local potential} for the metric. As a result, constant curvature problems for K\"ahler metrics boil down to scalar PDEs for their potentials, a famous instance being K\"ahler metrics with constant Ricci curvature, known as \emph{K\"ahler-Einstein} metrics, whose local potentials satisfy a complex Monge-Amp\`ere equation. This was in fact a main motivation for the introduction of K\"ahler metrics in \cite{Kah}, where it was also noted that the complex Monge-Amp\`ere equation in question can be written as the Euler-Lagrange equation of a certain functional. 

In the present paper, we will more generally consider \emph{constant scalar curvature} K\"ahler metrics (cscK metrics for short) on a compact complex manifold $X$. K\"ahler metrics in a fixed cohomology class of $X$ are parametrized by a space $\cH$ of (global) K\"ahler potentials $u\in C^\infty(X)$, cscK metrics corresponding to solutions in $\cH$ of a certain fourth-order nonlinear elliptic PDE. Remarkably, the latter is again the Euler-Lagrange equation of a functional $M$ on $\cH$, discovered by T.~Mabuchi. While $M$ is generally not convex on $\cH$ as an open convex subset of $C^\infty(X)$, Mabuchi defined a natural Riemannian $L^2$-metric on $\cH$ with respect to which $M$ does become convex, opening the way to a variational approach to the cscK problem. The picture was further clarified by S.K.~Donaldson, who noted that $\cH$ behaves like an infinite dimensional symmetric space and emphasized the analogy with the log norm function in Geometric Invariant Theory. 

Using this as a guide, one would like to detect the growth properties of $M$ by looking at its slope at infinity along certain geodesic rays in $\cH$ arising from algebro-geometric one-parameter subgroups, and prove that positivity of these slopes ensures the existence of a minimizer, which would then be a cscK metric. This is basically the prediction of the \emph{Yau-Tian-Donaldson conjecture}, positivity of the algebro-geometric slopes at infinity being equivalent to \emph{K-stability}. In the K\"ahler-Einstein case, this conjecture was famously solved a few years ago by \cite{CDS}, thereby completing intensive research on positively curved K\"ahler-Einstein metrics with many key contributions by G.Tian. 

The more elementary case of convex functions on (finite dimensional) Riemannian symmetric spaces (see \ref{sec:convslope}) and experience from the direct method of the calculus of variations suggest to try to attack the general case of the conjecture along the following steps: 
\begin{enumerate}

\item extend $M$ to a convex functional on a certain metric completion $\bar\cH$, in which coercivity (\ie linear growth) implies the existence of a minimizer; 

\item prove that a minimizer $u$ of $M$ in $\bar\cH$ is a weak solution to the cscK PDE in some appropriate sense, and show that ellipticity of this equation implies that $u$ is smooth, hence a cscK potential; 

\item show that $M$ is either coercive, or bounded above on some geodesic ray in $\bar\cH$;   

\item approximate any geodesic ray $(u_t)$ in $\bar\cH$ by algebro-geometric rays $(u_{j,t})$ in $\cH$, in such a way that (uniform) positivity of the slopes of $M$ along $(u_{j,t})$ forces $M(u_t)\to+\infty$ at infinity. 

\end{enumerate}
As of this writing, (1) and (3) are fully understood, as a combination of \cite{Che00a,Dar15,DR,BerBer,BBEGZ,BBJ}. On the other hand, while (2) and (4) are known in the K\"ahler-Einstein case \cite{BBEGZ,BBJ}, they remain wide open in general\footnote{A proof of (2) has recently been announced by X.X.Chen and J.Cheng \cite{CC1,CC2,CC3}}. The goal of this text is to survey these developments, as well as the analysis of the algebro-geometric slopes at infinity in terms of non-Archimedean geometry, building on \cite{KT,siminag,nama}. It is organized as follows:

\begin{itemize}

\item \S\ref{sec:convnorms} describes the 'baby case' of convex functions on the space of Hermitian norms of a fixed vector space, introducing alternative Finsler metrics and the space of non-Archimedean norms as the cone at infinity; 

\item Section~\ref{sec:Kahler} recalls the basic formalism of K\"ahler potentials and energy functionals; 

\item Section~\ref{sec:var} reviews the link between the metric geometry of $\cH$ and pluripotential theory, and discusses (1), (2) and (3) above; 

\item Section~\ref{sec:NA} introduces the non-Archimedean counterparts to K\"ahler potentials and the energy functionals, and presents a proof of (4) in the K\"ahler-Einstein case. 

 \end{itemize}

\subsection*{Acknowledgments}
My current view of the subject has been framed by collaborations and countless discussions with many mathematicians. Without attempting to be exhaustive, I would like to thank in particular Bo Berndtsson, Jean-Pierre Demailly, Ruadha\'i Dervan, Dennis Eriksson, Philippe Eyssidieux, Paul Gauduchon, Henri Guenancia, Tomoyuki Hisamoto, Mihai P\u{a}un, Valentino Tosatti and David Witt Nystr\"om, and to express my profound gratitude to Robert Berman, Charles Favre, Vincent Guedj, Mattias Jonsson and Ahmed Zeriahi for the key role their ideas have been playing in our joint works over the years. 
%
%%%%%%%%%%%%%%%%%%%%%%%%%%%%%%%%%%%%%%%%%%%%%%%%%%%%%%%%%%%%%%%%%%%
%
\section{Convex functions on spaces of norms}\label{sec:convnorms}
The complexification $G$ of any compact Lie group $K$ is a reductive complex algebraic group, giving rise to a Riemannian symmetric space $G/K$ and a conical Tits building. The latter can be viewed as the asymptotic cone of $G/K$, and the growth properties of any convex, Lipschitz continuous function on $G/K$ are encoded in an induced function on the building. While this picture is well-known (see for instance \cite{KLM}), it becomes very explicit for the unitary group $U(N)$, for which $G/K\simeq\cN$ is the space of Hermitian norms on $\C^N$. The goal of this section is to discuss this case in elementary terms, along with alternative Finsler metrics on $\cN$, providing a finite dimensional version of the more sophisticated K\"ahler geometric setting considered afterwards. 

%
%%%%%%%%%%%%%%%%%%%%%%%%%%%%%%%%%%%%%%%%%%%%%%%%%%%%%%%%%%%%%%%%%%%
%
\subsection{Finsler geometry on the space of norms}\label{sec:Finslernorms}
Let $V$ be a complex vector space of finite dimension $N$, and denote by $\cN$ the space of Hermitian norms $\g$ on $V$, viewed as an open subset of the ($N^2$-dimensional) real vector space  $\Herm(V)$ of Hermitian forms $h$. The ordered spectrum of $h\in\Herm(V)$ with respect to $\g\in\cN$ defines a point $\la_\g(h)$ in the \emph{Weyl chamber} 
$$
\cC=\{\la\in\R^N\mid\la_1\ge\dots\ge\la_N\}\simeq\R^N/\fS_N, 
$$
where the symmetric group $\fS_N$ acts on $\R^N$ by permuting coordinates.
\begin{lem}\label{lem:subadd} For each symmetric $\mathrm{(}$\ie $\fS_N$-invariant$\mathrm{)}$ norm $\chi$ on $\R^N$, we have 
$$
\chi\left(\la_\g(h+h')\right)\le\chi\left(\la_\g(h)\right)+\chi\left(\la_\g(h')\right)
$$
for all $\g\in\cN$ and $h,h'\in\Herm(V)$.
\end{lem}
\begin{proof} Given $\la,\la'\in\cC$, one says that $\la$ is \emph{majorized by $\la'$}, written $\la\preceq\la'$, if 
$$
\la_1+\dots+\la_i\le\la'_1+\dots+\la'_i
$$ 
for all $i$, with equality for $i=N$. It is a well-known and simple consequence of the Hahn--Banach theorem that $\la\preceq\la'$ iff $\la$ is in the convex envelope of the $\fS_N$-orbit of $\la'$, which implies $\chi(\la)\le\chi(\la')$ by convexity, homogeneity and $\fS_N$-invariance of $\chi$. The Lemma now follows from the classical Ky Fan inequality $\la_\g(h+h')\preceq\la_\g(h)+\la_\g(h')$. 
\end{proof}
Thanks to \ref{lem:subadd}, setting $|h|_{\chi,\g}:=\chi(\la_\g(h))$ defines a continuous Finsler norm $|\cdot|_\chi$ on $\cN$, and hence a length metric $d_\chi$ on $\cN$, with $d_\chi(\g,\g')$ defined as usual as the infimum of the lengths $\int_0^1|\dot\g_t|_{\chi,\g_t}dt$ of all smooth paths $(\g_t)_{t\in[0,1]}$ in $\cN$ joining $\g$ to $\g'$. By equivalence of norms in $\R^N$, all metrics $d_\chi$ on $\cN$ are Lipschitz equivalent. 

\begin{exam} The metric $d_2$ induced by the $\ell^2$-norm on $\R^N$ is the usual Riemannian metric of $\cN$ identified with the Riemannian symmetric space $\GL(N,\C)/U(N)$. In particular, $(\cN,d_2)$ is a complete $\CAT(0)$-space, a nonpositive curvature condition implying that any two points of $\cN$ are joined by a unique (length minimizing) geodesic. 
\end{exam}

\begin{exam} The metric $d_\infty$ induced by the $\ell^\infty$-norm on $\R^N$ admits a direct description as a sup-norm
$$
d_\infty(\g,\g')=\sup_{v\in V\setminus\{0\}}\left|\log\g(v)-\log\g'(v)\right|,
$$
whose exponential is the best constant $C>0$ such that $C^{-1}\g\le\g'\le C\g$ on $V$. 
\end{exam}

In order to describe the geometry of $(\cN,d_\chi)$, introduce for each basis $e=(e_1,\dots,e_N)$ of $V$ the embedding 
$$
\iota_e:\R^N\hookrightarrow\cN
$$ 
that sends $\la\in\R^N$ to the Hermitian norm for which $e$ is orthogonal and $e_i$ has norm $e^{-\la_i}$. The image $\iota_e(\R^N)$ is thus the set of norms in $\cN$ that are diagonalized in the given basis $e$. Any two $\g,\g'\in\cN$ can be jointly diagonalized in some basis $e$, \ie $\g=\iota_e(\la),\g'=\iota_e(\la')$ with $\la,\la'\in\R^N$. After permutation, the vector $\la'-\la$ determines an element $\la(\g,\g')\in\cC$ which only depends on $\g,\g'$, and is obtained by applying $-\log$ to the spectrum of $\g'$ with respect to $\g$. The following result, proved in \cite{BE}, generalizes the well-known Riemannian picture for $d_2$. 

\begin{thm}\label{thm:dchi} For each symmetric norm $\chi$ on $\R^N$, the induced Finsler metric $d_\chi$ on $\cN$ is given by $d_\chi(\g,\g')=\chi\left(\la(\g,\g')\right)$ for all $\g,\g'\in\cN$. It is further characterized as the unique metric on $\cN$ such that $\iota_e:(\R^N,\chi)\hookrightarrow(\cN,d_\chi)$ is an isometric embedding for all bases $e$. 
\end{thm}

%
%%%%%%%%%%%%%%%%%%%%%%%%%%%%%%%%%%%%%%%%%%%%%%%%%%%%%%%%%%%%%%%%%%%
\subsection{Convergence to non-Archimedean norms}\label{sec:NAnorm}
By a \emph{geodesic ray} $(\g_t)_{t\in\R_+}$ in $\cN$, we mean a constant speed Riemannian geodesic ray, \ie $d_2(\g_t,\g_s)$ is a constant multiple of $|t-s|$. Every geodesic ray is of the form $\g_t=\iota_e(t\la)$ for some basis $e$ and $\la\in\R^N$, the latter being uniquely determined up to permutation as the spectrum of the Hermitian form $\dot\g_t$ with respect to $\g_t$ for any value of $t$. As a result, $(\g_t)$ is also a (constant speed) geodesic ray for all Finsler metrics $d_\chi$, and indeed satisfies $d_\chi(\g_t,\g_s)=\chi(\la)|t-s|$. The metric $d_\chi$ might admit other geodesic rays in general, but we will not consider these in what follows. 

Two geodesic rays $(\g_t),(\g'_t)$ are called \emph{asymptotic} if $\g_t$ and $\g'_t$ stay at bounded distance with respect to any of the Lipschitz equivalent metrics $d_\chi$, \ie are uniformly equivalent as norms on $V$. This defines an equivalence relation on the set of geodesic rays, whose quotient naturally identifies with a space of \emph{non-Archimedean norms}. 

To see this, pick a geodesic ray $\g_t=\iota_e(t\la)$. Then $\g_t(v)^2=\sum_i|v_i|^2 e^{-2\la_i t}$ for each vector $v=\sum_i v_i e_i$ in $V$, from which one easily gets that $\g_t(v)^{1/t}$ converges to 
\begin{equation}\label{equ:NAdiag}
\a\left(\sum_i v_i e_i\right):=\max_{v_i\ne 0}e^{-\la_i}.
\end{equation}
as $t\to\infty$. The function $\a:V\to\R_+$ so defined satisfies 
\begin{itemize}
\item[(i)] $\a(v+v')\le\max\{\a(v),\a(v')\}$; 
\item[(ii)] $\a(\tau v)=\a(v)$ for all $\tau\in\C^*$; 
\item[(iii)] $\a(v)=0\Longleftrightarrow v=0$, 
\end{itemize}
which means that $\a$ is an element of the space $\cN^\NA$ of non-Archimedean norms on $V$ with respect to the \emph{trivial absolute value} $|\cdot|_0$ on the ground field $\C$, \ie $|0|_0=0$ and $|\tau|_0=1$ for $\tau\in\C^*$. The closed balls of such a norm are linear subspaces of $V$, and the data of $\a$ thus amounts to that of an $\R$-filtration of $V$, or equivalently a flag of linear subspaces together with a tuple of real numbers; for this reason, $\cN^\NA$ is also known in the literature as the \emph{(conical) flag complex}. The space $\cN^\NA$ has a natural $\R_+^*$-action $(t,\a)\mapsto\a^t$, whose only fixed point is the \emph{trivial norm} $\a_0$ on $V$.

The existence of a basis of $V$ compatible with a given flag implies that any non-Archimedean norm $\a\in\cN^\NA$ can be diagonalized in some basis $e=(e_i)$, in the sense that it satisfies \ref{equ:NAdiag} for some $\la\in\R^N$. The image of $\la$ in $\R^N/\fS_N$ is uniquely determined by $\a$, and a complete invariant for the (non-transitive) action of $G=\GL(V)$ on $\cN^\NA$, inducing an identification
$$
\cN^\NA/G\simeq\R^N/\fS_N.
$$
The structure of $\cN^\NA$ can be analyzed just as that of $\cN$ by introducing for each basis $e$ the embedding 
$$
\iota^\NA_e:\R^N\hookrightarrow\cN^\NA
$$ 
sending $\la\in\R^N$ to the non-Archimedean norm (\ref{equ:NAdiag}). Any two norms can be jointly diagonalized, \ie belong to the image of $\iota_e$ for some $e$, and it is proved in \cite{BE} that there exists a unique metric $d^\NA_\chi$ on $\cN^\NA$ for which each $\iota^\NA_e:(\R^N,\chi)\to(\cN^\NA,d^\NA_\chi)$ is an isometric embedding. It is worth mentioning that the Lipschitz equivalent metric spaces $(\cN^\NA,d^\NA_\chi)$, while complete, are \emph{not} locally compact as soon as $N>1$. 

\begin{exam}\label{exam:1PS} Every (algebraic) $1$-parameter subgroup $\rho:\C^*\to\GL(V)$ defines a non-Archimedean norm $\a_\rho\in\cN^\NA$, characterized by 
$$
\a_\rho(v)\le r\Longleftrightarrow\lim_{\tau\to 0}\tau^{\lceil\log r\rceil}\rho(\tau)\cdot v\text{ exists in }V.
$$ 
If $e=(e_i)$ is a basis of eigenvectors for $\rho$ with $\rho(\tau)\cdot e_i=\tau^{\la_i}e_i$, $\la_i\in\Z$, then $\a_\rho=\iota_e(\la)$. This shows that the lattice points $\cN^\NA_\Z$, \ie the images of $\Z^N$ by the embeddings $\iota_e$,  are exactly the norms attached to $1$-parameter subgroups, and ultimately leads to an identification of $(\cN^\NA,d_2)$ with the (conical) Tits building of the reductive algebraic group $\GL(V)$. 
\end{exam}

Coming back to geodesic rays, one proves that the non-Archimedean norms $\a=\lim\g_t^{1/t}$, $\a'=\lim\g_t'^{1/t}$ defined by two rays $(\g_t)$, $(\g'_t)$ are equal iff the rays are asymptotic, and that $d_\chi^\NA$ computes the slope at infinity of $d_\chi$, \ie
\begin{equation}\label{equ:slopedist}
d_\chi^\NA(\a,\a')=\lim_{t\to\infty}\frac{d_\chi\left(\g_t,\g'_t\right)}{t}.
\end{equation}
%
%
%%%%%%%%%%%%%%%%%%%%%%%%%%%%%%%%%%%%%%%%%%%%%%%%%%%%%%%%%%%%%%%%%%%
%
%
\subsection{Slopes at infinity of a convex function}\label{sec:convslope}
If $f:\R_+\to\R$ is convex,\linebreak $(f(t)-f(0))/t$ is a nondecreasing function of $t$. The \emph{slope at infinity}
$$
f'(\infty):=\lim_{t\to+\infty}\frac{f(t)}{t}\in(-\infty,+\infty]
$$ 
is thus well-defined, and finite if $f$ is Lipschitz continuous. It is characterized as the supremum of all $s\in\R$ such that $f(t)\ge s t+O(1)$ on $\R_+$, and $f$ is bounded above iff $f'(\infty)\le 0$. 

A function $F:\cN\to\R$ on the space of Hermitian norms is (geodesically) convex iff $F\circ\iota_e:\R^N\to\R$ is convex for each basis $e$, and similarly for a function on $\cN^\NA$. Assume further that $F$ is Lipschitz. Then $F(\g_t)$ is convex and Lipschitz continuous on $\R_+$ for each geodesic ray $\g$, and the slope at infinity $\lim_{t\to+\infty} F(\g_t)/t$ only depends on the equivalence class $\a\in\cN^\NA$ defined by $\g$. As a result, $F$ determines a function 
$$
F^\NA:\cN^\NA\to\R,
$$ 
characterized by $F(\g_t)/t\to F^\NA(\a)$ for each ray $(\g_t)$ asymptotic to $\a\in\cN^\NA$, and this function is further convex and Lipschitz continuous by \ref{equ:slopedist}. 

\begin{thm}\label{thm:coerc} Let $F:\cN\to\R$ be a convex, Lipschitz continuous function, and fix a base point $\g_0\in\cN$ and a symmetric norm $\chi$ on $\R^N$. The following are equivalent: 
\begin{itemize}
\item[(i)] $F:\cN\to\R$ is an \emph{exhaustion function}, \ie proper and bounded below; 
\item[(ii)] $F$ is \emph{coercive}, \ie $F(\g)\ge\d\,d_\chi(\g,\g_0)-C$ for some constants $\d,C>0$; 
\item[(iii)] $F^\NA(\a)>0$ for all nontrivial $\a\in\cN^\NA$; 
\item[(iv)] there exists $\d>0$ such that $F^\NA\ge\d\,d_\chi^\NA$.  
\end{itemize}
These conditions are further satisfied as soon as $F$ admits a unique minimizer. 
\end{thm}

\begin{proof} Clearly, (ii) implies (i), and (i) implies that $F(\g_t)$ is unbounded for any geodesic ray, hence has a positive slope at infinity, which yields (iii). Let us now prove (iii)$\Longrightarrow$(ii). Assuming by contradiction that there exists a sequence $\g_j$ in $\cN$ such that
\begin{equation}\label{equ:noncoer}
F(\g_j)\le\d_j d_\chi(\g_j,\g_0)-C_j
\end{equation}
with $\d_j\to 0$ and $C_j\to+\infty$, we are going to construct a non-constant geodesic ray $(\g_t)$ along which $F$ is bounded above, contradicting the positivity of the slope at infinity along this ray. By Lipschitz continuity, \ref{equ:noncoer} implies $T_j:=d_\chi(\g_j,\g_0)\to\infty$. For each $j$, let $(\g_{j,t})_{t\in [0,T_j]}$ be the geodesic segment joining $\g_0$ to $\g_j$, parametrized so that $t=d_\chi(\g_{j,t},\g_0)$. By Ascoli's theorem, $(\g_{j,t})$ converges to a geodesic ray $(\g_t)$, uniformly on compact sets of $\R_+$. By convexity of $F$, we have 
$$
\frac{F(\g_{j,t})-F(\g_0)}{t}\le\frac{F(\g_j)-F(\g_0)}{T_j}, 
$$
hence $F(\g_{j,t})\le\d_jt+F(\g_0)$, which yields in the limit the upper bound $F(\g_t)\le F(\g_0)$. At this point, we have thus shown that (i), (ii) and (iii) are equivalent. That (ii)$\Longrightarrow$(iv) follows from \ref{equ:slopedist}, while (iv) clearly implies (iii). 

Assume finally that $F$ admits a unique minimizer, which we may take as the base point $\g_0$. If $F$ is not coercive, the previous argument yields a nonconstant ray $(\g_t)$ such that $F(\g_t)\le F(\g_0)=\inf F$, which shows that all $\g_t$ are mininizers of $F$, and hence $\g_t=\g_0$ by uniqueness, a contradiction.  \end{proof}

\section{The constant scalar curvature problem for K\"ahler metrics}\label{sec:Kahler}
This section recalls the basic formalism of constant curvature K\"ahler metrics, and introduces the corresponding energy functionals. 
%
%
%%%%%%%%%%%%%%%%%%%%%%%%%%%%%%%%%%%%%%%%%%%%%%%%%%%%%%%%%%%%%%%%%%%
%
%
\subsection{K\"ahler metrics with constant curvature}\label{sec:curv}
Let $X$ be a compact complex manifold, and denote by $n$ its (complex) dimension. The data of a Hermitian metric on the tangent bundle $T_X$ is equivalent to that of a positive $(1,1)$-form $\om$, locally expressed in holomorphic coordinates $(z_j)$ as 
$\om=\sqrt{-1}\sum_{ij}\om_{ij} dz_i\wedge d\bar z_j$ with $(\om_{ij})$ a smooth family of positive definite Hermitian matrices. One says that $\om$ is \emph{K\"ahler} if it satisfies the following equivalent conditions: 
\begin{itemize}
\item[(i)] $d\om=0$; 
\item[(ii)] $\om$ admits local potentials, \ie smooth real valued valued functions $u$ such that $\om=\sqrt{-1}\ddbar u$, or 
$\om_{ij}=\partial^2 u/\partial z_i\partial\bar z_j$ in local coordinates; 
\item[(iii)] the Levi-Civita connection $\nabla$ of $\om$ on the tangent bundle $T_X$ coincides with the Chern connection, \ie the unique Hermitian connection with $\nabla^{0,1}=\overline{\partial}$. 
\end{itemize}
The K\"ahler condition thus ensures compatibility between Riemannian and complex Hermitian geometry. The (normalized) curvature tensor $\Theta_\om(T_X):=\frac{\sqrt{-1}}{2\pi} \nabla^2$ of a K\"ahler metric is a $(1,1)$-form with values in the Hermitian endomorphisms of $T_X$, whose trace with respect to $T_X$ coincides with the \emph{Ricci curvature} $\Ric(\om)$ in the sense of Riemannian geometry. In other words, the Ricci tensor of a K\"ahler metric can be seen as the curvature of the induced metric on the dual of the \emph{canonical bundle} $K_X:=\det T^\star_X$, the factor $2\pi$ being included in the curvature so that the de Rham cohomology class of the closed $(1,1)$-form $\Ric(\om)$ coincides with the first Chern class 
$$
c_1(X):=c_1(T_X)=-c_1(K_X).
$$
In terms of the normalized operator $\ddc :=\frac{\sqrt{-1}}{2\pi}\ddbar$ and a local function $u$ with $\om=\ddc  u$, we have
$$
\Ric(\om)=-\ddc \log\det\left(\frac{\partial^2 u}{\partial z_j\partial\bar z_k}\right),
$$
which accounts for the ubiquity of the \emph{complex Monge-Amp\`ere operator} $u\mapsto\det\left(\partial^2 u/\partial z_j\partial\bar z_k\right)$ in K\"ahler geometry. Taking the trace of $\Ric(\om)$ with respect to $\om$ yields the \emph{scalar curvature} 
$$
S(\om)=n\frac{\Ric(\om)\wedge\om^{n-1}}{\om^n}=\D\log\det\left(\frac{\partial^2 u}{\partial z_j\partial\bar z_k}\right).
$$
Denote by $V:=\int_X\om^n=[\om]^n$ the volume of $\om$, and observe that the mean value of $S(\om)$ is the cohomological constant 
$$
V^{-1}\int_X S(\om)\om^n=n V^{-1}\int_X\Ric(\om)\wedge\om^{n-1}=-n\la
$$ 
with 
$$
\la:=V^{-1}\left(c_1(K_X)\cdot[\om]^{n-1}\right). 
$$
As a result, there exists a unique function $\rho\in C^\infty(X)$, the \emph{Ricci potential} of $\om$, such that
$$
\left\{
    \begin{array}{ll}
        \D\rho =S(\om)+n\la \\
        \int_X e^\rho\om^n=1.
    \end{array}
\right.
$$
This defines a smooth, positive probability measure $\mu_0:=e^\rho\om^n$ which we call the \emph{Ricci normalized volume form of $\om$}. 

To the above three notions of curvature correspond the following three versions of the constant curvature problem.
\begin{itemize}
\item[(a)] Requiring the full curvature tensor of $\om$ to be constant, \ie 
$$
\Theta_\om(T_X)=-\frac{\la}{n}\om\otimes\id_{T_X},
$$
is a very strong condition which implies uniformization, in the sense that $(X,\om)$ must be isomorphic (after scaling the metric) to the complex projective space ($\la<0$), a finite quotient of a compact complex torus ($\la=0$), or a cocompact quotient of the complex hyperbolic ball ($\la>0$). 

\item[(b)] A \emph{K\"ahler-Einstein metric} (KE for short) is a K\"ahler metric $\om$ of constant Ricci curvature, \ie satisfying $\Ric(\om)=-\la\om$, the K\"ahler analogue of the Einstein equation. Passing to cohomology classes yields the necessary proportionality condition 
\begin{equation}\label{equ:propor}
c_1(K_X)=\la[\om]
\end{equation}
in $H^2(X,\R)$, which implies that the canonical bundle has a sign: $X$ is either \emph{canonically polarized} ($\la>0$), \emph{Calabi-Yau} ($\la=0$) or \emph{Fano} ($\la<0$). 

\item[(c)] Finally, a \emph{constant scalar curvature K\"ahler metric} (cscK for short) is a K\"ahler metric $\om$ with $S(\om)$ constant, \ie $S(\om)=-n\la$. Here the sign of $\la$ only gives very weak information on the positivity properties of $K_X$. Note that $S(\om)$ is constant iff the Ricci potential $\rho$ is harmonic, hence constant by compactness of $X$.
\end{itemize}
While a KE metric $\om$ is trivially cscK, it is remarkable that the converse is also true as soon as the (necessary) cohomological proportionality condition holds, the reason being
\begin{equation}\label{equ:Riccipot}
(\ref{equ:propor})\Longrightarrow\Ric(\om)=-\la\om+\ddc \rho.
\end{equation}
This follows indeed from the $\ddbar$-lemma, which states that an exact $(p,q)$-form on a compact K\"ahler manifold is $\ddbar$-exact, hence \ref{equ:propor} $\Longleftrightarrow\Ric(\om)=-\la\om+\ddc  f$ for some $f\in C^\infty(X)$. Taking the trace with respect to $\om$ shows that $f-\rho$ is harmonic, hence constant, proving \ref{equ:Riccipot}. 

Thanks to the same $\ddbar$-lemma, one can introduce \emph{global} potentials for K\"ahler metrics in a fixed cohomology class. More precisely, given a K\"ahler form $\om$, any other K\"ahler form in the cohomology class of $\om$ is of the form $\om_u:=\om+\ddc  u$ with $u$ a \emph{K\"ahler potential}, \ie an element of the open, convex set of smooth functions
$$
\cH:=\left\{u\in C^\infty(X)\mid\om_u>0\right\}. 
$$
Assuming \ref{equ:propor}, and hence $\Ric(\om)=-\la\om+\ddc \rho$, a simple computation yields 
$$
\Ric(\om_u)+\la\om_u=\ddc \log\left(\frac{e^{\la u}\mu_0}{\om_u^n}\right),
$$
and $\om_u$ is thus K\"ahler-Einstein iff $u$ satisfies the complex Monge-Amp\`ere equation
\begin{equation}\label{equ:MAKE}
\MA(u):=V^{-1}\om_u^n=c\,e^{\la u}\mu_0
\end{equation}
where $c>0$ is a normalizing constant ensuring that the right-hand side is a probability measure. 
%\begin{rmk}\label{rmk:twist} It turns out to be useful to consider more generally \emph{twisted K\"ahler-Einstein} and \emph{twisted cscK metrics}. The general idea is to add to the canonical class $K_X$ a closed $(1,1)$-form $\theta$, which amounts to replacing the above quantities with their twisted versions
%\begin{enumerate}
%\item $\Ric_\theta(\om):=\Ric(\om)-\theta$, 
%\item $S_\theta(\om):=\tr_\om S_\theta(\om)$, 
%\item $\la_\theta:=V^{-1}(c_1(K_X)+[\theta])\cdot[\om]^{n-1}$,
%\item $\mu_\theta:=e^{r_\theta}\om^n$ with $\D r_\theta=S_\theta(\om)+n\la_\theta$. 
%\end{enumerate} 
%Here again, twisted cscK and twisted KE metrics coincide if the cohomological proportionality condition $c_1(K_X)+[\theta]=\la_\theta[\om]$ holds. In this generality, twisted KE metrics correspond to solutions of the complex Monge-Amp\`ere \ref{equ:MAKE} with an arbitrary volume form in place of $\mu_0$. 
%\end{rmk}
%
%%%%%%%%%%%%%%%%%%%%%%%%%%%%%%%%%%%%%%%%%%%%%%%%%%%%%%%%%%%%%%%%%%%
%
\subsection{Energy functionals}\label{sec:energy}
A fundamental feature of the cscK problem, discovered by \cite{Mab86}, is that the corresponding (fourth order) PDE $S(\om_u)+n\la=0$ for a potential $u$ can be written as the Euler-Lagrange equation of a functional $M:\cH\to\R$, the \emph{Mabuchi K-energy functional}. It is characterized by 
$$
\frac{d}{dt} M(u_t)=-\int_X\dot u_t\left(S(\om_{u_t})+n\la\right)\MA(u_t)
$$
for any smooth path $(u_t)$ in $\cH$, and normalized by $M(0)=0$. Note that $M(u)$ is invariant under translation of a constant, hence only depends on the K\"ahler metric $\om_u$. The Chen--Tian formula for $M$ \cite{Che00b,Tianbook} yields a decomposition
$$
M=M_{\ent}+M_{\pp}, 
$$
where the \emph{entropy part} 
$$
M_{\ent}(u):=\int_X\log\left(\frac{\MA(u)}{\mu_0}\right)\MA(u)\in [0,+\infty)
$$
is the relative entropy of the probability measure $\MA(u)$ with respect to the Ricci normalized volume form $\mu_0$, and the \emph{pluripotential part} $M_{\pp}(u)$ is a linear combination of terms of the form 
$\int_Xu\,\om_u^j\wedge\om^{n-j}$ and  $\int_Xu\,\Ric(\om)\wedge\om_u^j\wedge \om^{n-j-1}$.

\smallskip

Assume now that the cohomological proportionality condition $c_1(K_X)=\la[\om]$ holds, so that $\om_u$ is cscK iff $u$ satisfies the complex Monge-Amp\`ere \ref{equ:MAKE}. Besides the K-energy $M$, another (simpler) functional also has \ref{equ:MAKE} as its Euler--Lagrange equation. Indeed, the complex Monge-Amp\`ere operator $\MA(u)$ is the derivative of a functional $E:\cH\to\R$, \ie
$$
\frac{d}{dt}E(u_t)=\int_X\dot u_t\MA(u_t)
$$
The functional $E$, normalized by $E(0)=0$, is called the \emph{Monge-Amp\`ere energy} (with strong fluctuations in both notation and terminology across the literature), and is explicitly given by
\begin{equation}\label{equ:E}
E(u)=\frac{1}{n+1}\sum_{j=0}^nV^{-1}\int_Xu\,\om_u^j\wedge\om^{n-j}.
\end{equation}
It follows that $\om_u$ is cscK (equivalently, KE) iff $u$ is a critical point of the \emph{Ding functional} $D:\cH\to\R$, defined as $D:=L-E$ with 
$$
L(u):=\left\{
    \begin{array}{ll}
        \la^{-1}\log\left(\int_X e^{\la u}\mu_0\right) & \mbox{ if }\la\ne0\\
        \int_X u\,\mu_0 & \mbox{ if }\la=0.
    \end{array}
\right.
$$
Note that $E(u+c)=E(u)+c$ and $L(u+c)=L(u)+c$ for $c\in\R$, so that $D(u)$, just as $M(u)$, is invariant under translation of $u$ by a constant, and hence only depends on the K\"ahler form $\om_u$. 

%\begin{rmk} This discussion again carries over to the twisted case, giving rise to the twisted K-energy $M_\theta$ and twisted Ding functional $D_\theta=L_\theta-E$, defined by replacing $S(\om)$, $\la$ and $\mu_0$ with $S_\theta(\om)$, $\la_\theta$ and $\mu_\theta$ as in \ref{rmk:twist}. In fact, 
%\begin{equation}\label{equ:Mtwisted}
%M_\theta=M+n V^{-1}\left[E_\theta-\left([\theta]\cdot[\om]^{n-1}\right)E\right]
%\end{equation}
%with $E_\theta$ the primitive of $u\mapsto\theta\wedge\om_u^{n-1}$. 
%\end{rmk}

%
%%%%%%%%%%%%%%%%%%%%%%%%%%%%%%%%%%%%%%%%%%%%%%%%%%%%%%%%%%%%%%%%%%%
%
%
\section{The variational approach}\label{sec:var}
This section first describes the $L^p$-geometry of the space of K\"ahler potentials, with respect to which the K-energy becomes convex. This is used to relate the coercivity of $M$, its growth along geodesic rays, and the existence of minimizers. 
%
%
%%%%%%%%%%%%%%%%%%%%%%%%%%%%%%%%%%%%%%%%%%%%%%%%%%%%%%%%%%%%%%%%%%%
%
%
\subsection{The Mabuchi $L^2$-metric and weak geodesics}
As we saw above, cscK metrics are characterized as critical points of the K-energy $M:\cH\to\R$. In order to set up a variational approach to the cscK problem, an ideal scenario would thus be that $M$ be convex with respect to the linear structure of $\cH$ as an open convex subset of the vector space $C^\infty(X)$, which would in particular imply that cscK metrics correspond to minimizers of $M$. 

While convexity in this sense fails in general, Mabuchi realized in \cite{Mab87} that $M$ does become convex with respect to a more sophisticated notion of geodesics in $\cH$. The infinite dimensional manifold $\cH$ is indeed endowed with a natural Riemannian metric, defined at $u\in\cH$ as the $L^2$-scalar product with respect to the volume form $\MA(u)=V^{-1}\om_u^n$. Mabuchi computed the Levi-Civita connection and curvature of this $L^2$-metric, and proved that the (Riemannian) Hessian of $M$ is everywhere nonnegative, so that $M$ is convex along (smooth) geodesics in $\cH$. 

The existence of a geodesic joining two given points in $\cH$ thus becomes a pressing issue, and new light was shed on this problem in \cite{Sem,Don99}, with the key observation that the equation for geodesics in $\cH$ can be rewritten as a complex Monge-Amp\`ere equation. In terms of the one-to-one correspondence between paths $(u_t)_{t\in I}$ of functions on $X$ parametrized by a open interval $I\subset\R$ and $S^1$-invariant functions $U$ on the product $X\times\DD_I$ of $X$ with the annulus 
$$
\DD_I:=\{\tau\in\C\mid -\log|\tau|\in I\}
$$
given by setting
\begin{equation}\label{equ:Phi}
U(x,\tau)=u_{-\log|\tau|}(x),
\end{equation}
a smooth path $(u_t)_{t\in I}$ in $\cH$ is a geodesic iff $U$ satisfies the complex Monge-Amp\`ere equation
\begin{equation}\label{equ:HCMA}
\left(\om+\ddc  U\right)^{n+1}=0. 
\end{equation}
Finding a geodesic $(u_t)_{t\in[0,1]}$ joining two given points $u_0,u_1\in\cH$ thus amounts to solving \ref{equ:HCMA} with prescribed boundary data. While uniqueness is a simple matter, existence is much more delicate (and turns out to fail in general), as vanishing of the right-hand side makes this nonlinear elliptic equation degenerate. Since the restriction of the $(1,1)$-form $\om+\ddc U$ to each slice $X\times\{\tau\}$ is required to be positive, \ref{equ:HCMA} imposes that $\om+\ddc U\ge 0$, which means by definition that $U$ is $\om$-psh (for plurisubharmonic). Thanks to this observation, geodesics can be approached using pluripotential theory. 

Denote by $\PSH(X,\om)$ the space of $\om$-psh functions on $X$, \ie pointwise limits of decreasing sequences in $\cH$, by \cite{BK}. Following \cite[\S 2.2]{Bern2}, we define a \emph{subgeodesic} in $\PSH(X,\om)$ as a family $(u_t)_{t\in I}$ of $\om$-psh functions whose corresponding function $U$ on $X\times\DD_I$ is $\om$-psh, a condition which implies in particular that $u_t(x)$ is a convex function of $t$. A \emph{weak geodesic} $(u_t)_{t\in I}$ is a subgeodesic which is \emph{maximal}, \ie for any compact interval $[a,b]\subset I$ and any subgeodesic $(v_t)_{t\in(a,b)}$, 
$$
\lim_{t\to a} v_t\le u_a\text{ and }\lim_{t\to b} v_t\le u_b\Longrightarrow v_t\le u_t\text{ for }t\in(a,b).
$$
\begin{lem}\label{lem:supaff}\cite{Dar17b} Let $(u_t)_{t\in I}$ be a weak geodesic in $\PSH(X,\om)$, and pick a compact interval $[a,b]\subset I$. If $u_b-u_a$ is bounded above, then $t\mapsto\sup_X(u_t-u_a)$ is affine on $[a,b]$. 
\end{lem}
\begin{proof} After reparametrizing, we assume for ease of notation that $a=0$ and $b=1$, and set $m:=\sup_X(u_1-u_0)$. For $t\in[0,1]$, the inequality $\sup_X(u_t-u_0)\le tm$ follows directly from the convexity of $t\mapsto u_t(x)$. Since $v_t(x):=u_1(x)+(t-1)m$ is a subgeodesic with $v_0\le u_0$ and $v_1\le u_1$, maximality of $(u_t)$ implies $v_t\le u_t$ for $t\in[0,1]$, and hence 
$$
tm=\sup_X(u_1-u_0)+(t-1)m\le\sup_X(u_t-u_0).
$$ 
\end{proof}
Given $u_0,u_1\in\PSH(X,\om)$, the weak geodesic $(u_t)_{t\in(0,1)}$ joining them is defined as the usc upper envelope of the family of all subgeodesics $(v_t)_{t\in(0,1)}$ such that $\lim_{t\to 0} v_t\le u_0$, $\lim_{t\to 1} v_t\le u_1$ (or $u_t\equiv-\infty$ if no such subgeodesic exists). When $u_0,u_1$ are bounded, the weak geodesic $(u_t)$ is locally bounded, and a 'balayage' argument shows that the corresponding function $U$ is the unique locally bounded solution to \ref{equ:HCMA} in the sense of \cite{BT76}, with the prescribed boundary data. Even for $u_0,u_1\in\cH$, exemples due to \cite{LV} show that the weak geodesic $(u_t)$ joining them is not $C^2$ in general, but initial work by \cite{Che00a}, succesively refined in \cite{Blo} and \cite{CTW}, eventually established that $U$ is locally $C^{1,1}$. 
%
%
%%%%%%%%%%%%%%%%%%%%%%%%%%%%%%%%%%%%%%%%%%%%%%%%%%%%%%%%%%%%%%%%%%%
%
%
\subsection{$L^p$-geometry in the space of K\"ahler potentials}
Just as the Riemannian metric on the space of norms $\cN$ can be generalized to a Finsler $\ell^p$-metric for any $p\in[1,\infty]$ (cf.~\ref{sec:Finslernorms}), it was noticed by T.~Darvas that the Mabuchi $L^2$-metric on $\cH$ admits an immediate generalization to an $L^p$-Finsler metric, by replacing the $L^2$-norm with the $L^p$-norm in the above definition. The associated pseudometric $d_p$ on $\cH$ is defined by letting $d_p(u,u')$ be the infimum of the $L^p$-lengths
$$
\int_0^1\|\dot u_t\|_{L^p(\MA(u_t))}dt
$$ 
of all smooth paths $(u_t)_{t\in[0,1]}$ in $\cH$ joining $u$ to $u'$. We trivially have $d_p\le d_{p'}$ for $p\le p'$, but the fact that $d_p$ is actually a metric (\ie separates distinct points) is a nontrivial result in this infinite dimensional setting, proved in \cite{Che00a} for $p=2$ and in \cite{Dar15} for $d_1$, and hence for all $d_p$. 

The space $\cH$ is not complete for any of the metrics $d_p$, and the description of the completion was completely elucidated in \cite{Dar15} in terms of pluripotential theory, following an earlier attempt by V.~Guedj. The class 
$$
\cE\subset\PSH(X,\om)
$$
of $\om$-psh functions $u$ with \emph{full Monge-Amp\`ere mass}, introduced by Guedj-Zeriahi in \cite{GZ} (see also \cite{BEGZ}), may be described as the largest class of $\om$-psh functions on which the Monge-Amp\`ere operator $u\mapsto\MA(u)$ is defined and satisfies: 
\begin{itemize}
\item[(i)] $\MA(u)$ is a probability measure that puts no mass on pluripolar sets, \ie sets of the form $\{\p=-\infty\}$ with $\p$ $\om$-psh; 
\item[(ii)] the operator is continuous along decreasing sequences. 
\end{itemize}

For $p\in[1,\infty]$, the class $\cE^p\subset\cE$ of $\om$-psh functions with \emph{finite $L^p$-energy} is defined as the set of $u\in\cE$ that are $L^p$ with respect to $\MA(u)$. For domains in $\C^n$, the analogue of $\cE^p$ was first introduced by U.~Cegrell in his pioneering work \cite{Ceg}. 

\begin{exam} If $X$ is a Riemann surface, a function $u\in\PSH(X,\om)$ belongs to $\cE$ iff the measure $\om+\ddc  u$ puts no mass on polar sets, and $u$ is in $\cE^1$ iff it satisfies the classical finite energy condition $\int_X du\wedge d^c u<+\infty$, which means that the gradient of $u$ is in $L^2$. 
\end{exam}

The following results are due to T.~Darvas. 

\begin{thm}\cite{Dar15}\label{thm:Dar} The metric $d_p$ admits a unique extension to $\cE^p$ that is continuous along decreasing sequences, and $(\cE^p,d_p)$ is the completion of $(\cH,d_p)$. Further:
\begin{itemize}
\item[(i)] $d_p(u,u')$ is Lipschitz equivalent to $\|u-u'\|_{L^p(\MA(u))}+\|u-u'\|_{L^p(\MA(u'))}$; 
\item[(ii)] the weak geodesic $(u_t)_{t\in[0,1]}$ joining any two $u_0,u_1\in\cE^p$ is contained in $\cE^p$, and is a constant speed geodesic in the metric space $(\cE^p,d_p)$, \ie $d_p(u_t,u_{t'})=c|t-t'|$ for some constant $c$. 
\end{itemize}
\end{thm}
%
%
%%%%%%%%%%%%%%%%%%%%%%%%%%%%%%%%%%%%%%%%%%%%%%%%%%%%%%%%%%%%%%%%%%%
%
%
\subsection{Energy functionals on $\cE^1$}
The weakest metric $d_1$ turns out to be the most relevant one for K\"ahler geometry, due to its close relationship with the Monge-Amp\`ere energy $E$. By \cite{BBGZ,Dar15}, mixed Monge-Amp\`ere integrals of the form 
$$
\int_X u_0\,\om_{u_1}\wedge\dots\wedge\om_{u_n}
$$ 
with $u_i\in\cE^1$ are well-defined, and continuous with respect to the $u_i$ in the $d_1$-topology. In particular, the Monge-Amp\`ere operator is continuous in this topology, and \ref{equ:E} yields a continuous extension of $E$ to $\cE^1$, which is proved to be convex on subgeodesics, and \emph{affine} on weak geodesics. 

\begin{lem}\label{lem:d1E} If $u,u'\in\cE^1$ satisfy $u\le u'$, then $d_1(u,u')=E(u')-E(u)$. 
\end{lem}
\begin{proof} By monotone regularization, it is enough to prove this for $u,u'\in\cH$. The corresponding weak geodesic $(u_t)_{t\in[0,1]}$ is then $C^{1,1}$, and its $L^1$-length $\int_{t=0}^1\int_X|\dot u_t|\MA(u_t)$ computes $d_1(u,u')$. By \ref{lem:supaff}, $u_t(x)$ is a nondecreasing function of $t$, hence $\dot u_t\ge 0$, which yields 
$$
d_1(u,u')=\int_0^1dt\int_X\dot u_t\MA(u_t)=\int_0^1\left(\frac{d}{dt} E(u_t)\right)dt=E(u')-E(u).
$$
\end{proof}
When dealing with translation invariant functionals such as $M$ and $D$, it is useful to introduce the translation invariant functional $J:\cE^1\to\R_+$ defined by 
$$
J(u):=V^{-1}\int_X u\,\om^n-E(u),
$$
which vanishes iff $u$ is constant and satisfies $J(u)=d_1(u,0)+O(1)$ on functions normalized by $\sup u=0$, thanks to \ref{lem:d1E}. 

%For any two $,u'\in\cE^1$, we have $|E(u)-E(u')|\le d_1(u,u')$, with equality if $\le'$ or $'\le$. In general, there exists a largest function $P(u,u')$ in $\cE^1$ bounded above by both $$ and $'$, and we then have 
%$$
%d_1(u,u')=d_1(u,P(u,u'))+d_1(P(u,u'),u), 
%$$
%which yields 
%$$
%d_1(u,u')=E(u)+E(u')-2E(P(u,u')). 
%$$
Since the pluripotential part $M_\pp$(u) of the Mabuchi K-energy is a linear combination of integrals of the form $\int_X u\,\om_u^j\wedge\om^{n-j}$ and $\int_Xu\,\Ric(\om)\wedge\om_u^j\wedge\om^{n-j-1}$, it admits a continuous extension $M_\pp:\cE^1\to\R$. As to the entropy part $M_\ent$, it extends to a lower semicontinuous functional 
$$
M_\ent:\cE^1\to[0,+\infty],
$$
by defining $M_\ent(u)$ to be the relative entropy of $\MA(u)$ with respect to $\mu_0$. Finiteness of $M_\ent(u)$ is a subtle condition, which amounts to saying that $\MA(u)$ has a density $f$ with respect to Lebesgue measure such that $f\log f$ is integrable.

\begin{thm}\label{thm:convcomp}\cite{BBEGZ,BerBer,CLP,BDL2} The extended functionals satisfy the following properties. 
\begin{itemize} 
\item[(i)] For each $C>0$, the set of $u\in\cE^1$ with $\sup_X u=0$ and $M_\ent(u)\le C$ is compact in the $d_1$-topology.
\item[(ii)] $|M_\pp(u)|\le A J(u)+B$ for some constant $A,B>0$. 
\item[(iii)] The functional $M:\cE^1\to(-\infty,+\infty]$ is lower semicontinuous and convex on weak geodesics. 
\end{itemize}
\end{thm}

%
%
%%%%%%%%%%%%%%%%%%%%%%%%%%%%%%%%%%%%%%%%%%%%%%%%%%%%%%%%%%%%%%%%%%%
%
%
\subsection{Variational characterization of cscK metrics}
The Mabuchi K-energy $M$ is \emph{coercive} if $M\ge\d J-C$ on $\cE^1$ by for some constants $\d,C>0$. By \cite{BDL2}, it is in fact enough to test this on $\cH$. We then have the following basic dichotomy. 

\begin{thm}\cite{DH,DR,BBJ}\label{thm:Me1} If the K-energy $M$ is coercive, then it admits a minimizer in $\cE^1$. If not, then for any $u\in\cH$, there exists a unit speed weak geodesic ray $(u_t)_{t\in[0,+\infty)}$ in $\cE^1$ emanating from $u$, normalized by $\sup_X(u_t-u)=0$, along which $M(u_t)$ is nonincreasing. 
\end{thm}
\begin{proof} Assume that $M$ is coercive, and let $u_j\in\cE^1$ be a minimizing sequence, which can be normalized by $\sup u_j=0$ by translation invariance. Since $M(u_j)$ is bounded above, $J(u_j)$ is bounded, by coercivity, hence so is $|M_\pp(u_j)|\le A J(u_j)+B$. As a result, $M_\ent(u_j)$ is also bounded, which means that $u_j$ stays in a compact subset of $\cE^1$. After passing to a subsequence, we may thus assume that $u_j$ admits a limit $u\in\cE^1$, which is a minimizer of $M$ by lower semicontinuity. 

Assume now that $M$ is not coercive, \ie $M(u_j)\le\d_j J(u_j)-C_j$ for some sequences $u_j\in\cE^1$ with $\sup (u_j-u)=0$, $\d_j\to 0$ and $C_j\to+\infty$. We then argue as in \ref{thm:coerc}. Since $M_\ent(u_j)\ge 0$ and $M_\pp(u_j)\ge -A J(u_j)-B$, $(A+\d_j) J(u_j)\ge C_j-B$ tends to $\infty$, hence so does
$$
T_j:=d_1(u_j,u)=J(u_j)+O(1).
$$ 
Denote by $(u_{j,t})_{t\in [0,T_j]}$ the weak geodesic connecting $u$ to $u_j$, parametrized so that $d_1(u_{j,t},u_{j,s})=|t-s|$, and note that $\sup_X(u_{j,t}-u)=0$ for all $t$, by \ref{lem:supaff}. By convexity of $M$ along $(u_{j,t})$, we get
\begin{equation}\label{equ:Mconv}
\frac{M(u_{j,t})-M(u)}{t}\le\frac{M(u_j)-M(u)}{T_j}\le\d_j.  
\end{equation}
for $j\gg 1$. For each $T>0$ fixed, $|M_\pp(u_{j,t})|\le AJ(u_{j,t})+B$ is bounded for $t\le T$, hence so is $M_\ent(u_{j,t})$, by \ref{equ:Mconv}. By \ref{thm:convcomp}, the $1$-Lipschitz maps $t\mapsto u_{j,t}$ thus send each compact subset of $\R_+$ to a fixed compact set in $\cE^1$, and Ascoli's theorem shows that $(u_{j,t})$ converges uniformly on compact sets of $\R_+$ to a ray $(u_t)_{t\in\R_+}$ in $\cE^1$ (after passing to a subsequence). By local uniform convergence, $(u_t)$ is a weak geodesic, and satisfies $\sup(u_t-u)=0$ and $d_1(u_t,u_s)=|t-s|$. Further, $M(u_t)\le M(u)$ by \ref{equ:Mconv} and lower semicontinuity, which implies that $M(u_t)$ decreases, by convexity. 
\end{proof}

Using their key convexity result and a perturbation argument, Berman-Berndtsson proved in \cite{BerBer} that cscK metrics in the class $[\om]$ minimize $M$, and that the identity component $\Aut^0(X)$ of the group of holomorphic automorphisms acts transitively on these metrics. In \cite{BDL}, Berman-Darvas-Lu went further and proved that the existence of \emph{one} cscK metric $\om_u$ implies that any other minimizer of $M$ lies in the $\Aut^0(X)$-orbit of $u$, and hence is smooth. Using this, we have: 

 \begin{cor}\cite{DR,BDL}\label{cor:Me1} If $\Aut^0(X)$ is trivial and $M$ admits a minimizer $u\in\cH$, then $M$ is coercive. 
\end{cor}
\begin{proof} By \cite{BDL}, $u$ is the unique minimizer of $M$ in $\cE^1$, up to a constant. Assume by contradiction that $M$ is not coercice, and let $(u_t)$ be the ray constructed in \ref{thm:Me1}. Since $M(u_t)\le M(u)=\inf M$, $u_t$ must be equal to $u$ up to a constant, and hence $u_t=u$ by normalization, which contradicts $d_1(u_t,u)=t$. 
\end{proof}

If a minimizer $u$ of $M$ lies in $\cH$, then $u+t f$ is in $\cH$ for all test functions $f\in C^\infty(X)$ and $0<t\ll 1$, hence $M(u+t f)\ge M(u)$, which implies that $u$ is a critical point of $M$, \ie $\om_u$ is cscK. This simple perturbation argument cannot be performed for a minimizer in $\cE^1$, which is a major remaining difficulty\footnote{This difficulty has recently been overcome in \cite{CC1,CC2,CC3}.} on the analytic side of the cscK problem. In the K\"ahler-Einstein case, we have however: 

\begin{thm}\label{thm:varKE}\cite{BBGZ,BBEGZ} If the cohomological proportionality condition, \ref{equ:propor} holds, any mimizer of $M$ in $\cE^1$ lies in $\cH$, and hence defines a K\"ahler-Einstein metric. 
\end{thm}
\begin{proof} It is not hard to show that a minimizer for $M$ is also a minimizer for the Ding functional $D=L-E$, whose critical points in $\cH$ are solutions of the complex Monge-Amp\`ere \ref{equ:MAKE}. The main step is now to prove that a minimizer $u\in\cE^1$ of $D$ satisfies \ref{equ:MAKE} in the sense of pluripotential theory, for the complex Monge-Amp\`ere arsenal can then be used to infer ultimately that $u$ is smooth. The projection argument to follow goes back to Aleksandrov in the setting of real Monge-Amp\`ere equations. Given a test function $f\in C^\infty(X)$, the \emph{psh envelope} $P(u+f)$ is defined as the largest $\om$-psh function dominated by $u+f$. The functional $L$ makes sense on any function $u$, $\om$-psh or not, and satisfies $u\le v\Longrightarrow L(u)\le L(v)$. We thus get for each $t>0$ 
$$
L(u)-E(u)=D(u) \le D(P(u+t f))\le L(u+t f)-E(P(u+t f)).
$$
The key ingredient is now a differentiability result proved in \cite{BB}, which implies that $t\mapsto E(P(u+t f))$ is differentiable at $0$, with derivative equal to $\int_X f\,\MA(u)$. This yields indeed 
$$
\int_Xf\,\MA(u)=\lim_{t\to 0_+}\frac{E(P(u+t f))-E(u)}{t}\le\lim_{t\to 0_+}\frac{L(u+t f)-L(u)}{t}=\frac{\int_X f\,e^{\la u}\mu_0}{\int_X e^{\la u}\om_0},
$$
which proves, after replacing $f$ with $-f$, that $u$ is a weak solution of \ref{equ:MAKE}. 
\end{proof}

%
%%%%%%%%%%%%%%%%%%%%%%%%%%%%%%%%%%%%%%%%%%%%%%%%%%%%%%%%%%%%%%%%%%%
%
%
\section{Non-Archimedean K\"ahler geometry and K-stability}\label{sec:NA}
In this final section, we turn to the non-Archimedean aspects of the cscK problem. We reformulate K-stability as a positivity property for the non-Archimedean analogue of the K-energy $M$, and explain how uniform K-stability implies coercivity, in the K\"ahler-Einstein case. 
%
%%%%%%%%%%%%%%%%%%%%%%%%%%%%%%%%%%%%%%%%%%%%%%%%%%%%%%%%%%%%%%%%%%%
%
%
\subsection{Non-Archimedean pluripotential theory}
If $X$ is a complex algebraic variety, we denote by $X^\NA$ its \emph{Berkovich analytification} (viewed as a topological space) with respect to the \emph{trivial absolute value} $|\cdot|_0$ on $\C$ \cite{Berko}. When $X=\spec A$ is affine, with $A$ a finitely generated $\C$-algebra, $X^\NA$ is defined as the set of all multiplicative seminorms $|\cdot|:A\to\R_+$ compatible with $|\cdot|_0$, endowed with the topology of pointwise convergence. In the general case, $X$ can be covered by finitely many affine open sets $X_i$, and $X^\NA$ is defined by gluing together the analytifications $X_i^\NA$ along their common open subsets $(X_i\cap X_j)^\NA$. 

Assume from now on that $X$ is projective, equipped with an ample line bundle $L$. The topological space $X^\NA$ is then compact (Hausdorff), and can be viewed as a compactification of the space of real-valued valuations $v:\C(X)^*\to\R$ on the function field of $X$, identifying $v$ with the multiplicative norm $|\cdot|=e^{-v}$. In particular, the trivial valuation on $\C(X)$ defines a special point $0\in X^\NA$, fixed under the natural $\R_+^*$-action $(t,|\cdot|)\mapsto |\cdot|^t$. 

In this trivially valued setting, (the analytification of) $L$ comes with a canonical \emph{trivial metric}. Any section $s\in H^0(X,L)$ thus defines a continuous function $|s|_0:X^\NA\to [0,1]$, the value of $-\log|s|_0$ at a valuation $v$ being equal to that of $v$ on the local function corresponding to $s$ in a trivialization of $L$ at the center of $v$. 

The space $\cH^\NA$ of \emph{non-Archimedean K\"ahler potentials} (with respect to $L$) is defined as the set of continuous functions $\f\in C^0(X^\NA)$ of the form 
$$
\f=\frac{1}{k}\max_i\{\log|s_i|_0+\la_i\}
$$
with $(s_i)$ a finite set of sections of $H^0(kL)$ without common zeroes and $\la_i\in\R$, those with $\la_i\in\Q$ forming $\cH^\NA_\Q\subset\cH^\NA$. In order to motivate this definition, recall that the data of a Hermitian norm $\g$ on $H^0(kL)$ defines a Fubini-Study/Bergman type metric on $L$, whose potential with respect to a reference metric $|\cdot|_0$ on $L$ can be written as 
$$
\FS_k(\g):=\frac{1}{k}\log\max_{s\in H^0(kL)\setminus\{0\}}\frac{|s|_0}{\g(s)}=\frac{1}{2k}\log\sum_i |s_i|_0^2
$$
for any $\g$-orthonormal basis $(s_i)$. Similarly, any non-Archimedean norm $\a$ on $H^0(kL)$ in the sense of \ref{sec:NAnorm} admits an orthogonal basis $(s_i)$, and we then have 
$$
\FS^\NA_k(\a):=\frac{1}{k}\log\max_{s\in H^0(kL)\setminus\{0\}}\frac{|s|_0}{\a(s)}=\frac{1}{k}\max_i\{\log|s_i|_0+\la_i\}.
$$
with $\la_i=-\log\a(s_i)$. Denoting respectively by $\cN_k$ and $\cN_k^\NA$ the spaces of Hermitian and non-Archimedean norms on $H^0(kL)$, we thus have two natural maps 
$$
\FS_k:\cN_k\to\cH,\,\,\,\,\,\FS^\NA_k:\cN_k^\NA\to\cH^\NA,
$$
and $\cH^\NA=\bigcup_k\FS^\NA_k(\cN^\NA_k)$ by definition. This is to be compared with the fact that $\bigcup_k \FS_k(\cN_k)$ is dense in $\cH$, a consequence of the fundamental Bouche--Catlin--Tian--Zelditch asymptotic expansion of Bergman kernels.

Non-Archimedean K\"ahler potentials are closely related to \emph{test configurations} for $(X,L)$, \ie $\C^*$-equivariant partial compactifications $(\cX,\cL)\to\C$ of the product $(X,L)\times\C^*$, with $\cL$ a $\Q$-line bundle. 

\begin{prop}\label{prop:tc} Every test configuration $(\cX,\cL)$ gives rise in a natural way to a function $\f_\cL\in C^0(X^\NA)$, which belongs to $\cH_\Q^\NA$ if $\cL$ is ample, and is a difference of functions in $\cH^\NA_\Q$ in general. Further, two test configurations $(\cX_i,\cL_i)$, $i=1,2$ yield the same function on $X^\NA$ if and only if $\cL_1$ and $\cL_2$ coincide after pulling-back to some higher test configuration.  
\end{prop}
To define $\f_\cL$, denote respectively by $\cL'$ and $L_{\cX'}$ the pullbacks of $\cL$ and $L$ to the graph $\cX'$ of the canonical $\C^*$-equivariant birational map $\cX\dashrightarrow X\times\C$, and note that 
$$
\cL'=L_{\cX'}+D
$$ 
for a unique $\Q$-Cartier divisor $D$ supported in the central fiber $\cX'_0$. Every valuation $v$ on $X$ admits a natural $\C^*$-invariant (Gauss) extension $G(v)$ to $\C(X)(t)\simeq\C(\cX')$, which can be evaluated on $D$ by chosing a local equation for (a Cartier multiple of) $D$ at the center of $G(v)$, and we set $\f_\cL(v):=G(v)(D)$.

\begin{exam} Every $1$-parameter subgroup $\rho:\C^*\to\GL(H^0(kL))$ with $kL$ very ample defines a test configuration $(\cX,\cL)$, obtained as the closure of the orbit of $X\hookrightarrow\P H^0(kL)^*$. The $\Q$-line bundle $\cL$ is ample, and every test configuration $(\cX,\cL)$ with $\cL$ ample arises this way. By \ref{exam:1PS}, $\rho$ also defines a non-Archimedean norm $\a_\rho$ on $H^0(kL)$, and we have 
$$
\f_\cL=\FS^\NA_k(\a_\rho).
$$
Combined with \ref{prop:tc}, this implies that $\cH^\NA_\Q$ is in one-to-one correspondence with the set of all normal, ample test configurations. 
\end{exam}

A more general \emph{$L$-psh function} is defined as a usc function $\f:X^\NA\to[-\infty,+\infty)$ that can be written as the pointwise limit of a decreasing sequence (or net, rather) in $\cH^\NA$, defining a space $\PSH^\NA$. These functions are bounded above, and the maximum principle takes the simple form 
$$
\sup_{X^\NA}\f=\f(0),
$$
with $0\in X^\NA$ the trivial valuation. The space $\PSH^\NA$ is endowed with a natural topology of pointwise convergence on divisorial points, in which functions with $\sup\f=0$ form a compact set. This is proved in \cite{trivval}, building on previous work \cite{siminag} dealing with Berkovich spaces over the field $\C{\lau t}$ of formal Laurent series. 

\begin{exam}\label{exam:ideal} If $\fa$ is a coherent ideal sheaf on $X$, setting $|\fa|=\max_{f\in\fa} |f|$ defines a continuous function $|\fa|:X^\NA\to[0,1]$. Given $c>0$, one shows that the function $c\log|\fa|$ is $L$-psh if and only if $L\otimes\fa^c$ is nef, in the sense that $\mu^*L-c E$ is nef on the normalized blow-up $\mu:X'\to X$ of $\fa$, with exceptional divisor $E$. 
\end{exam}
%
%
%%%%%%%%%%%%%%%%%%%%%%%%%%%%%%%%%%%%%%%%%%%%%%%%%%%%%%%%%%%%%%%%%%%
%
%
\subsection{From geodesic rays to non-Archimedean potentials}
We assume from now on that $X$ is a projective manifold equipped with an ample line bundle $L$, and $\om\in c_1(L)$ is a K\"ahler form. Recall that a subgeodesic ray $(u_t)_{t\in\R_+}$ in $\PSH(X,\om)$ is encoded in the associated $S^1$-invariant $\om$-psh function 
$$
U(x,\tau)=u_{-\log|\tau|}(x)
$$ 
on $X\times\DD^*$. We shall say that $(u_t)$ has \emph{linear growth} if $u_t\le a t+b$ for some constants $a,b\in\R$, \ie $U+a\log|\tau|\le b$, a condition which automatically holds when $(u_t)$ is a weak geodesic ray emanating from $u_0\in\cH$, as a consequence of \ref{lem:supaff}. 

To a subgeodesic ray $(u_t)$ with linear growth, we shall associate an $L$-psh function 
$$
U^\NA:X^\NA\to[-\infty,+\infty),
$$
following a procedure initiated in \cite{hiro}. Imposing 
$$
(U+a\log|\tau|)^\NA=U^\NA-a,
$$
we may assume that $U$ is bounded above, and hence extends to an $\om$-psh function on $X\times\DD$. Consider first the case where $U$ has \emph{analytic singularities}, \ie locally satisfies
$$
U=c\log\max_i|f_i|+O(1)
$$ 
for a fixed constant $c>0$ and finitely many holomorphic functions $(f_i)$.  The (integrally closed) ideal sheaf
$$
\fa:=\left\{f\in\cO_{X\times\DD}\mid c\log|f|\le U+O(1)\right\}
$$
is then coherent, and $\C^*$-invariant by $S^1$-invariance of $U$. We thus have a weight decomposition $\fa=\sum_{i=0}^r\tau^i\fa_i$ for an increasing sequence of coherent ideal sheaves $\fa_0\subset\fa_1\subset\dots\subset\fa_r=\cO_X$ on $X$. One further proves that $L\otimes\fa_i^c$ is nef for each $i$, yielding $L$-psh functions $c\log|\fa_i|$ on $X^\NA$ by \ref{exam:ideal}, and hence an $L$-psh function 
$$
U^\NA=c\log|\fa|:=c\max_i\{\log|\fa_i|-i\}.
$$
In the general case, the \emph{multiplier ideals} $\cJ(k U)$, $k\in\N^*$, are $\C^*$-invariant coherent ideal sheaves on $X\times\DD$. They satisfy the fundamental subadditivity property
$$
\cJ\left((k+k')U\right)\subset\cJ(k U)\cdot\cJ(k' U),
$$
which implies the existence of the pointwise limit
$$
U^\NA:=\lim_{k\to\infty}\tfrac 1 k\log|\cJ(k U)|
$$
on $X^\NA$. A variant of Siu's uniform generation theorem \cite[\S 3.2]{BBJ} further shows the existence of $k_0$ such that $p_X^*\left((k+k_0)L\right)\otimes\cJ(k U)$ is globally generated on $X\times\DD$ for all $k$, and it follows that $U^\NA$ is indeed $L$-psh. 

\begin{exam} Pick $k\gg 1$, and let $\g_t=\iota_s(t\la)$ be a geodesic ray in $\cN_k$ associated to a basis $s=(s_i)$ of $H^0(kL)$ and $\la\in\R^N$. The image $u_t:=\FS_k(\g_t)$ is then a subgeodesic ray in $\cH$ with linear growth, and $U^\NA=\FS_k^\NA\left(\iota^\NA_s(\la)\right)$ is the image of the non-Archimedean norm defined by $(\g_t)$. If $\la$ is further rational, $U$ has analytic singularities, and blowing-up $X\times\C$ along the associated $\C^*$-invariant ideal $\fa$ defines a test configuration $(\cX,\cL)$ such that $U^\NA=\f_\cL$. 
\end{exam}

The function $U^\NA$ basically captures the Lelong numbers of $U$, and we have in particular $U^\NA=0$ iff $U$ has zero Lelong numbers at all points of $X\times\{0\}$. More specifically, let $\cX$ be a normal test configuration for $X$, pick an irreducible component $E$ of the central fiber $\cX_0$, and set $b_E:=\ord_E(\cX_0)$. The normalized $\C^*$-invariant valuation $b_E^{-1}\ord_E$ on $\C(\cX)\simeq\C(X)(\tau)$ restricts to a divisorial (or trivial) valuation on $\C(X)$, defining a point $x_E\in X^\NA$. By \cite[Theorem 4.6]{BHJ1}, every divisorial point in $X^\NA$ is of this type, which means that $U^\NA$ is determined by its values on such points, and Demailly's work on multiplier ideals shows that 
$$
-b_E U^\NA(x_E)=\lim_{k\to\infty}\tfrac 1 k \ord_E\left(\cJ(k U)\right)
$$ 
coincides with the generic Lelong number along $E$ of the pull-back of $U$ (cf.~\cite[Proposition 5.6]{hiro}).

%While it was recently established by Y.Hashimoto that $\FS_k$ is injective for $k\gg 1$, this is never the case for $\FS^\NA_k$, since for each basis $s=(s_i)$ of $H^0(kL)$ we have 
%$$
%\FS_k^\NA(\iota^\NA_s(\la))=\frac{1}{k}\max_i\{\log|s_i|+\la_i\}=0
%$$
%as soon as $\max_i\la_i=0$ and the $s_i$ with $\la_i=0$ have no common zeroes. Equivalently, the ray $(u_t)$ in $\cH$ defined by
%$$
%u_t=\FS_k(\iota_s(t\la))=\frac{1}{2k}\log\sum_i |s_i|^2 e^{2\la_i t}
%$$ 
%is both nonconstant and bounded. 

%
%
%%%%%%%%%%%%%%%%%%%%%%%%%%%%%%%%%%%%%%%%%%%%%%%%%%%%%%%%%%%%%%%%%%%
%
%
\subsection{Non-Archimedean energy functionals}
Ideally, we would like to associate to each functional $F$ in \ref{sec:energy} a non-Archimedean analogue $F^\NA$, in such a way that 
\begin{equation}\label{equ:Fslope}
F^\NA(U^\NA)=\lim_{t\to\infty}\frac{F(u_t)}{t}
\end{equation} 
for all weak geodesic rays $(u_t)$. To get started, a special case of the pioneering work of A.Chambert-Loir and A.Ducros on forms and currents in Berkovich geometry \cite{CLD} enables to define a mixed non-Archimedean Monge-Amp\`ere operator
\begin{equation}\label{equ:mixedMANA}
(\f_1,\dots,\f_n)\mapsto\MA(\f_1,\dots,\f_n)
\end{equation}
on $n$-tuples $(\f_i)$ in $\cH^\NA$, with values in \emph{atomic} probability measures on $X^\NA$. When the $\f_i$ arise from test configurations $(\cX_i,\cL_i)$, we can assume after pulling back that all $\cX_i$ are equal to the same $\cX$, and we then have 
$$
\MA(\f_1,\dots,\f_n)=\sum_E b_E(\cL_1|_E\cdot...\cdot\cL_n|_E)\d_{x_E},
$$
where $\cX_0=\sum_E b_E E$ is the irreducible decomposition and the $x_E\in X^\NA$ are the associated divisorial points. 

We next introduce the \emph{non-Archimedean Monge-Amp\`ere energy} $E^\NA:\cH^\NA\to\R$ using the analogue of \ref{equ:E}. As in the K\"ahler case, $E^\NA$ is nondecreasing,  hence extends by monotonicity to $\PSH^\NA$, which defines a space 
$$
\cE^{1,\NA}:=\{E^\NA>-\infty\}\subset\PSH^\NA
$$ 
of \emph{$L$-psh functions $\f$ with finite $L^1$-energy}. It is proved in \cite{nama,trivval} that the mixed Monge-Amp\`ere operator \ref{equ:mixedMANA} admits a unique extension to $\cE^{1,\NA}$ with the usual continuity property along monotonic sequences, and that 
$$
J^\NA(\f):=\sup\f-E^\NA(\f)\in[0,+\infty)
$$
vanishes iff $\f\in\cE^1$ is constant. 
\begin{exam} A test configuration $(\cX,\cL)\to\C$, being a product away from the central fiber, admits a natural compactification $(\bar\cX,\bar\cL)\to\P^1$. The non-Archimedean Monge-Amp\`ere energy $E^\NA(\f)$ of the corresponding function $\f=\f_\cL\in\cH^\NA_\Q$ is then equal to the self-intersection number $\left(c_1(\bar\cL)^{n+1}\right)$, up to a normalization factor. Alternatively, 
$$
E^\NA(\f)=\lim_{k\to\infty}\frac{w_k}{k h^0(kL)}
$$
with $w_k\in\Z$ the weight of the induced $\C^*$-action on the determinant line $\det H^0(\cX_0,k\cL_0)$, see for instance \cite[\S 7.1]{BHJ1}.  
\end{exam}

If $(u_t)$ is a weak geodesic ray in $\cE^1$, $E(u_t)=a t+b$ is affine. Using that $U$ is more singular than $\cJ(kU)^{1/k}$, one shows that
\begin{equation}\label{equ:Eslope}
E^\NA(U^\NA)\ge a=\lim_{t\to\infty} E(u_t)/t, 
\end{equation}
which implies in particular that $U^\NA$ belongs to $\cE^{1,\NA}$. However, this inequality can be strict in general without further assumptions. 

\begin{exam}\label{exam:Eslope} Let $\om$ be the Fubini-Study metric on $X=\P^1$, normalized to mass $1$. A compact, polar Cantor set $K\subset\P^1$ carries a natural probability measure without atoms, and the potential $u$ of this measure with respect to $\om$ is smooth outside $K$, has zero Lelong numbers and does not belong to $\cE$. By \cite{RWN,Dar17b}, $u$ defines a locally bounded weak geodesic ray $(u_t)$ emanating from $0$ such that $E(u_t)=a t$ with $a<0$. However, the corresponding $\om$-psh function $U$ on $X\times\DD$ has zero Lelong numbers, hence $U^\NA=0$ and $E^\NA(U^\NA)=0$. 
\end{exam} 

The Mabuchi K-energy $M$ and the Ding functional $D$ also admit non-Archimedean analogues $M^\NA$ and $D^\NA$. While the pluripotential part $M^\NA_\pp$ of $M^\NA$ is defined in complete analogy with $M_\pp$ as a linear combination of mixed Monge-Amp\`ere integrals, the entropy part $M^\NA_\ent$ as well as $L^\NA$ turn out to be of a completely different nature, involving the \emph{log discrepancy function} 
$$
A_X:X^\NA\to[0,+\infty].
$$
The latter is the maximal lower semicontinuous extension of the usual log discrepancy on divisorial valuations, and we then have 
$$
M_\pp^\NA(\f)=\int_{X^\NA} A_X\MA(\f)
$$
and
$$
L^\NA(\f)=\left\{
    \begin{array}{ll}
        \la^{-1}\inf_{X^\NA}(A_X+\la\f) & \mbox{ if }\la\ne0\\
        \sup_{X^\NA}\f=\f(0) & \mbox{ if }\la=0.
    \end{array}
\right.
$$
where we have set as before $\la=V^{-1}(K_X\cdot L^{n-1})$. 
\begin{exam}\cite{BHJ1} If $(\cX,\cL)$ is an ample test configuration, then $M^\NA(\f)$ coincides with the \emph{Donaldson-Futaki} invariant of $(\cX,\cL)$, up to a nonnegative error term that vanishes precisely when $\cX_0$ is reduced. Further, $(X,L)$ is \emph{K-semistable} iff $M^\NA(\f)\ge 0$ for all $\f\in\cH_\Q^\NA$, and \emph{K-stable} iff equality holds only for $\f$ a constant. Following \cite{BHJ1,Der}, we say that $(X,L)$ is \emph{uniformly K-stable} if $M^\NA\ge\d J^\NA$ on $\cH_\Q^\NA$ for some $\d>0$. 
\end{exam} 
We can now state the following result, which builds in part on previous work by \cite{PRS} and \cite{BerKpoly}. 

\begin{thm}\cite{BHJ2,BBJ}\label{thm:slopeen} Let $(u_t)$ be any subgeodesic ray in $\cE^1$, normalized by $\sup u_t=0$. 
\begin{itemize}
\item[(i)] If $(u_t)$ has analytic singularities, then \ref{equ:Fslope} holds for $E$ and $M_\pp$.
\item[(ii)] If $(u_t)$ has strongly analytic singularities, then \ref{equ:Fslope} holds for $M_\ent$. 
\item[(iii)] In the Fano case, \ref{equ:Fslope} holds for $L$. 
\end{itemize}
\end{thm}
Here we say that $(u_t)$ (or $U$) has \emph{strongly analytic singularities} if $U$ satisfies near each point of $X\times\{0\}$ 
$$
U=\tfrac c 2\log\sum_i|f_i|^2\text{ mod }C^\infty
$$ 
for a fixed constant $c>0$ and finitely many holomorphic functions $(f_i)$. 
%
%
%%%%%%%%%%%%%%%%%%%%%%%%%%%%%%%%%%%%%%%%%%%%%%%%%%%%%%%%%%%%%%%%%%%
%
%
\subsection{A version of the Yau--Tian--Donaldson conjecture}
In its usual formulation, the Yau--Tian--Donaldson conjecture states that $c_1(L)$ contains a cscK metric if and only if $(X,L)$ is K-(poly)stable. In the following form, it says that $M$ satisfies the analogue of \ref{thm:coerc}. 

\begin{conj}\label{conj:YTD} Let $(X,L)$ be a polarized projective manifold, $\om\in c_1(L)$ be a K\"ahler form, and assume that $\Aut^0(X,L)=\C^*$. The following are equivalent: 
\begin{itemize}
\item[(i)] there exists a cscK metric in $c_1(L)$; 
\item[(ii)] $M$ is coercive; 
\item[(iii)] $(X,L)$ is uniformly K-stable. 
\end{itemize}
\end{conj} 
The implications (i)$\Longrightarrow$(ii)$\Longrightarrow$(iii) were respectively proved in \cite{BDL} (cf.~\ref{cor:Me1}) and \cite{BHJ2} (cf.~\ref{thm:slopeen}). By \ref{thm:Me1}, (ii) implies the existence of a minimizer $u\in\cE^1$ for $M$, and the key obstacle to get (i) is then to establish that $u$ is smooth\footnote{As already mentioned, this has recently been overcome in \cite{CC1,CC2,CC3}}. Assume now that (iii) holds. If (ii) fails, \ref{thm:Me1} yields a weak geodesic ray $(u_t)$ in $\cE^1$, emanating from $0$ and normalized by $\sup u_t=0$, $E(u_t)=-t$, along which $M(u_t)$ decreases, and hence $\lim M(u_t)/t\le 0$. Two major difficulties arise: 
\begin{enumerate}
\item While $U^\NA$ belongs to $\cE^{1,\NA}$, we cannot prove at the moment of this writing that (iii) propagates to $M^\NA\ge\d J^\NA$ on the whole of $\cE^{1,\NA}$. 
\item Even taking (1) for granted, \ref{exam:Eslope} shows that $M^\NA(U^\NA)$ cannot be expected to compute exactly the slope at infinity of $M(u_t)$. 
\end{enumerate} 
These difficulties can be overcome in the K\"ahler-Einstein case, by relying on the Ding functional as well. 

\begin{thm}\label{thm:ytd}\cite{BBJ} \ref{conj:YTD} holds if the proportionality condition $c_1(K_X)=\la[\om]$ is satisfied.
\end{thm}
\begin{proof}[Sketch of proof] For $\la\ge 0$, all three conditions in the conjecture are known to be always satisfied, and we thus focus on the Fano case. \ref{thm:varKE} completes the proof of (ii)$\Longrightarrow$(i), which was anyway proved long before \cite{Tianbook} by using Aubin's continuity method. Assume (iii), and consider a ray $(u_t)$ as above. In the Fano case, we have $M\ge D$, which shows that $D(u_t)=L(u_t)-E(u_t)$ is bounded above as well. We infer from \ref{thm:slopeen} that $\f:=U^\NA$ satisfies
$$
L^\NA(\f)=\lim_{t\to\infty}\frac{L(u_t)}{t}\le\lim_{t\to\infty}\frac{E(u_t)}{t}=-1\le E^\NA(\f).
$$
Relying on the Minimal Model Program along the same lines as \cite{LX}, one proves on the other hand that (iii) implies $D^\NA\ge\d J^\NA$ on $\cH^\NA$, and then on $\cE^{1,\NA}$ as well. As $\f$ is normalized by $\sup\f=0$, this means 
$$
L^\NA(\f)\ge (1-\d) E^\NA(\f)\ge\d-1,
$$
a contradiction.
\end{proof}

%
%
%%%%%%%%%%%%%%%%%%%%%%%%%%%%%%%%%%%%%%%%%%%%%%%%%%%%%%%%%%%%%%%%%%%
%
%

\end{document}